\documentclass[letterpaper, 11pt,  reqno]{amsart}

\usepackage[margin=1.1in,marginparwidth=1.5cm, marginparsep=0.5cm]{geometry}

\usepackage{amsmath,amssymb,amscd,amsthm,amsxtra, esint, stmaryrd,mathtools}
\usepackage{nccmath} 

\usepackage{mathrsfs} 

\usepackage[implicit=true]{hyperref}

\usepackage{color} 
\usepackage{ulem}
\usepackage[makeroom]{cancel}

\allowdisplaybreaks[2]

\definecolor{gr}{rgb}   {0.,   0.69,   0.23 }
\definecolor{bl}{rgb}   {0.,   0.5,   1. }
\definecolor{mg}{rgb}   {0.85,  0.,    0.85}
\definecolor{yl}{rgb}   {0.8,  0.7,   0.}
\definecolor{or}{rgb}  {0.7,0.2,0.2}

\newtheorem{theorem}{Theorem} [section]

\newtheorem{lemma}[theorem]{Lemma}
\newtheorem{proposition}[theorem]{Proposition}
\newtheorem{remark}[theorem]{Remark}

\DeclareMathOperator{\Id}{Id}

\newcommand{\noi}{\noindent}
\newcommand{\Z}{\mathbb{Z}}
\newcommand{\R}{\mathbb{R}}

\newcommand{\T}{\mathbb{T}}

\let\P= \undefined
\newcommand{\P}{\mathbf{P}}
\newcommand{\prob}{\mathbb{P}}

\newcommand{\E}{\mathbb{E}}

\newcommand{\al}{\alpha}
\newcommand{\be}{\beta}

\newcommand{\eps}{\varepsilon}

\newcommand{\g}{\gamma}
\newcommand{\G}{\Gamma}

\newcommand{\Ld}{\Lambda}
\newcommand{\s}{\sigma}

\newcommand{\ft}{\widehat}

\newcommand{\wt}{\widetilde}
\newcommand{\cj}{\overline}
\newcommand{\dx}{\partial_x}
\newcommand{\dt}{\partial_t}

\newcommand{\ta}{\theta}

\renewcommand{\l}{\ell}
\renewcommand{\o}{\omega}
\renewcommand{\O}{\Omega}

\newcommand{\les}{\lesssim}
\newcommand{\ges}{\gtrsim}

\newcommand{\jb}[1]
{\langle #1 \rangle}
\def\e{\varepsilon}

\newcommand{\N}{\mathbb{N}}

\newcommand{\NN}{\mathcal{N}}

\newcommand{\EE}{\mathcal{E}}

\newtheorem*{ackno}{Acknowledgements}

\numberwithin{equation}{section}
\numberwithin{theorem}{section}

\newcommand{\PP}{\mathbb{P}}

\newcommand{\vp}{\varphi}

\begin{document}
\baselineskip = 14pt

\title[Almost sure global for BBM]
{Almost sure global well-posedness for the Benjamin-Bona-Mahony equation}

\author[J.~Forlano]
{Justin Forlano}

\address{Justin Forlano,
School of Mathematics, Monash University, VIC 3800, Australia}

\email{justin.forlano@monash.edu}



\begin{abstract}
We prove that the Benjamin-Bona-Mahony equation is almost surely globally well-posed with respect to random Gaussian initial data of negative Sobolev regularity in $H^{s}(\T)$ for any $s>-\frac 14$.  This result is sharp in view of the mild probabilistic ill-posedness due to Oh-Tzvetkov (2026).
This also improves on a previous result of the author which established this only in a logarithmically negative regularity. To break through this logarithmic regularity barrier, we combine the $I$-method with the low-high argument from Bona-Tzvetkov (2007), which motivates a refined first-order expansion for solutions.
\end{abstract}

\maketitle
%

\section{Introduction}

We consider 
the Benjamin-Bona-Mahony equation (BBM): 
\begin{align}
\begin{cases}
\partial_{t}u-\partial_{xxt}u+\partial_{x}u+\frac{1}{2}\partial_{x}(u^{2})  = 0\\
u|_{t = 0} = u_0, 
\end{cases}
\ (x, t) \in \T \times \R_{+},
\label{BBM0}
\end{align}

\noi
where $u:\T \times \R_{+} \mapsto \R$ and $\T:=\R / (2\pi \Z)$ is the circle\footnote{We could also consider the BBM equation \eqref{BBM0} on $\T\times \R$ because of the time-reversal symmetry $u(x,t)\mapsto u(-x,-t)$ (viewing $\T$ as $[-\pi,\pi)$). However, for simplicity, we consider only positive times in the following.}.
The BBM equation is a model for the propagation of long wavelength, short amplitude water waves~\cite{Peregrine, BBM}.
 In particular, in \cite{BBM}, it was proposed as an alternative to the Korteweg-de Vries (KdV) equation.
 This is in part due to the boundedness of the dispersion relation for BBM while the dispersion relation for KdV is unbounded. For further discussion on the physical validity of the BBM model, see for example \cite{Model1, Model2, Model3}. See also \cite{HJY} which considered a rigorous limit problem from BBM to KdV on $\R$.

Our goal in this paper is to study the global well-posedness of BBM~\eqref{BBM0} in the low regularity setting. 
To this end, we rewrite \eqref{BBM0} in a form amenable for our analysis. We first remove the linear term $\dx u$ by translating vertically:
\begin{align}
 u(t,x) \longmapsto \wt{u}(t,x) =  u(t,x) + 1 \label{translate}
\end{align}
so that $\wt{u}$ now solves 
\begin{align}
\begin{cases}
\partial_{t}\wt{u}-\partial_{xxt}\wt{u}+\frac{1}{2}\partial_{x}(\wt{u}^{2})  = 0\\
\wt{u}|_{t = 0} = u_0+1.
\end{cases}
\label{BBM1}
\end{align}
As we work on the bounded set $\T$, the map \eqref{translate} is a bijection from $H^s(\T)$ to $H^{s}(\T)$ for any $s\in \R$. From now on, we focus entirely on \eqref{BBM1} rather than \eqref{BBM0}, and thus relabel $\wt{u}$ as $u$. 

The reason for the change of unknown \eqref{translate} is that it simplifies the dependence on time for solutions to \eqref{BBM1}. In particular, by factoring $\dt u - \dx^2 \dt u= (1-\dx^2)\dt u$, we may rewrite \eqref{BBM1} as 
\begin{align*}
\dt u = -\tfrac{1}{2} \vp(\dx) ( u^2)
\end{align*}
which becomes the integral equation
\begin{align}
u(t)  =    u_0 -\frac12  \int_{0}^{t} \vp(\dx) ( u(t')^2 )dt'  . \label{BBM}
\end{align}
where $\vp(\dx):= (1-\dx^2)^{-1}\dx$.  This formulation removes the need for a linear propagator $e^{t \vp(\dx)}$ appearing in \eqref{BBM} and this observation will simplify our analysis later on.

Regarding well-posedness, \eqref{BBM} is globally well-posed in $L^2(\T)$ and in $L^2(\R)$, see \cite{BCS1, BCS2, BonaTzvet, Roumegoux}. 
This result is sharp as~\eqref{BBM0} is ill-posed in negative Sobolev spaces in the sense that the solution map fails to be continuous in $H^s(\T)$ for any $s<0$ \cite{Panthee, BonaNormInf, FBBM}, through the phenomenon of infinite loss of regularity at general initial data \cite{BH}.

As the failure of continuity may be due to only certain exceptional bad  approximations, it does not rule out the existence of some unique solutions in negative regularity obeying some approximation property. 
In this vein, motivated by \cite{BTlocal, BTglobal, CollOh}, the author \cite{FBBM} previously studied the almost sure well-posedness problem for \eqref{BBM0} with random initial data below $L^2(\T)$. Namely,  random initial data of the form\footnote{We drop the factor of $2\pi$ as it plays no role in our analysis.} 
\begin{equation}
u_{0}^{\omega}(x)=\sum_{n\in \Z}\frac{g_{n}(\omega)}{\jb{n}^{\alpha}}e^{inx}, \label{initialdata}
\end{equation}
where $\al\in \R$, $\jb{\, \cdot\,}:=\sqrt{1+|\cdot|^{2}}$ and  $\{g_{n}\}_{n\in \Z}$ is a sequence of independent standard complex-valued Gaussian random variables on a probability space $(\Omega, \mathcal{F}, \prob)$ satisfying  $g_{n}=\overline{g_{-n}}$ and $g_{0}$ is real.
This initial data has the regularity property:\footnote{Here, we use the notation $a-$ (respectively, $a+$) to denote $a-\eps$ (respectively, $a+\eps$), where $0<\eps \ll1$ is extremely small.}
\begin{align}
u_{0}^{\o}\in H^{\alpha-\frac{1}{2}-}(\T)\setminus H^{\alpha-\frac{1}{2}}(\T) \label{regu0}
\end{align}
 almost surely. In view of the deterministic global well-posedness, we are interested in the regime $\al \leq \tfrac 12$, which corresponds to initial data \textit{outside} of $L^2(\T)$.
 In \cite{FBBM}, the author showed that BBM~\eqref{BBM0} is almost surely locally well-posed when $\al>\frac 14$, and almost surely globally well-posed in the special case $\al=\frac 12$. This latter restriction arose from the transitional behaviour of growth rates of the variance of \eqref{initialdata}: namely, if $\P_{\leq N}$, $N\in \N$, is the sharp Fourier projection to frequencies $\{|n|\leq N\}$, then
 \begin{align}
\g_{N}(\al)= \E[ |\P_{\leq N}u_0(x)|^2 ] = \sum_{|n|\leq N} \frac{1}{\jb{n}^{2\al}}  \sim 
\begin{cases}
\log N \quad &\text{if} \quad \al=\frac 12, \\
N^{1-2\al} \quad &\text{if} \quad \al<\frac 12,
\end{cases}
\label{variance}
\end{align}
 for any $x\in \T$. We see that the variance diverges polynomially if $\al<\frac 12$, but only logarithmically when $\al=\frac 12$. As we discuss further below, the globalisation argument in \cite{FBBM} could only handle the logarithmic blow-up situation $\al=\frac 12$.
 
In this paper, we overcome this logarithmic threshold and extend the almost sure global well-posedness to any $\al>\frac 14$.  We state our main result for \eqref{BBM}. See Remark~\ref{RMK:BBM} below for a similar result for \eqref{BBM0}.

 \begin{theorem}\label{THM:ASGWP} 
Let $\al>\frac 14$ and  $0<\e \ll 1$. Then, the BBM equation \eqref{BBM} is almost surely globally well posed in $H^{\al-\frac 12-\eps}(\T)$, with respect to random initial data of the form 
\begin{align}
u_{0}^{\o}=\sum_{n\in \Z}\frac{g_{n}(\o)}{\jb{n}^{\al}}e^{inx}. \label{globaldata}
\end{align}
 More precisely, for almost every $\omega \in \Omega$, there exists a unique solution $u$ of \eqref{BBM} in \begin{align*}
u_{0}+C(\R; H^{\min(3\al-\frac 12 ,2\al)-\eps}(\T))\subset C(\R; H^{\al-\frac 12-\eps}(\T))
\end{align*}
with initial condition $u_{0}^{\omega}$ of the form \eqref{globaldata} and which satisfies 
\begin{align}
\| u(t) \|_{H^{\al-\frac 12-\eps}} \leq C(\o) (1+|t|)^{a} \label{polygrowth}
\end{align}
for some $a=a(\al)>0$ and $C(\o)>0$.
\end{theorem}

 As discussed below, the result of Theorem~\ref{THM:ASGWP} constitutes crossing the logarithmic regularity barrier in the $I-$method argument used in \cite{FBBM} (for random initial data), which was based of the method in \cite{Leonardo, GKOT} (for stochastic dispersive PDE).
The key to this improvement is a \textit{refined} first-order expansion: for every $T>0$, there is $M=M(T)>0$ satisfying $M\to \infty$ as $T\to \infty$ for which
 \begin{align}
u(t)  = \P_{>M}u_0  + w_{M}(t)+ v_{M}(t) \label{expansion}
\end{align}
 for each $t\in [0,T]$, where $\P_{>M}:= \Id - \P_{\leq M}$ and where, when $\al\leq \frac 12$,

\smallskip
\noi
(i) $\P_{>M}u_0$ is part of the rough random initial data which gets \textit{smaller} as $T \to \infty$ in $H^{\al-\frac 12-}(\T)$ (replace by $\P_{>M}S(t)u_0$ if the equation has a non-trivial linear group)

\smallskip
\noi
(ii) $w_{M}$ is a \textit{large} mesoscopic term with regularity $H^{3\al -\frac 12-}(\T)$

\smallskip
\noi
 (iii) $v_{M}$ is a \textit{small} macroscopic term in $H^{2\al-}(\T)$ and where $u_{M} : = \P_{>M}u_0 +v_M$ solves \eqref{BBM0} with data $\P_{>M}u_0$ 
 
 \smallskip
 \noi
 This refined expansion is more suitable for establishing global-in-time dynamics, as evidenced by Theorem~\ref{THM:ASGWP}, as the standard first-order expansion \cite{Bourgain1, mckean, dpd} $u(t) = u_0+y(t)$ is insufficient when $\al<\frac 12$; see the discussion below. 
 
 We also point out that the argument in \cite{FBBM} yielded a super-exponential in time growth bound on solutions ($\sim e^{t^{b}}$ for some $b\in \R$) when $\al=\frac 12$. Our argument here improves this to a polynomial growth bound \eqref{polygrowth}.
 
 Moreover, our result is sharp as it covers the full range for which a satisfactory probabilistic well-posedness theory holds. 
 Indeed, when $\al \leq \frac 14$, the variance blow-up phenomenon occurs due to the inability to construct the quadratic random object $\vp(\dx)( u_0^{2})$ as a distribution. See \cite[Section 3.2]{FBBM}. Recently, Li-Li-Oh-Tzvetkov \cite{LLOT} investigated the probabilistic well-posedness issue for \eqref{BBM0} beyond the variance blow-up corresponding to \textit{any} $\al\leq \frac 14$. They showed that by rescaling the initial data by replacing it by $\g_{N}^{-1/4}\, \P_{\leq N}u_0^{\o}$, then the corresponding sequence of solutions $\{u_{N}(t)\}$  converges in law to the (global) solution of the stochastic BBM forced by a derivative of the spatial white noise $\xi$:
 \begin{align}
 \dt u - \dx^{2} \dt u + \dx u +\dx(u^2) = \dx \xi, \quad u\vert_{t=0}=0. \label{BBMxi}
\end{align}
In this renormalised sense, the almost sure global well-posedness for \eqref{BBM0} with random initial data of the form \eqref{initialdata} can be extended beyond the variance blow-up. However, without weakening the initial data, Theorem~\ref{THM:ASGWP} is sharp. Recently, this variance blow-up phenomenon has been reinterpreted as a form of mild ill-posedness for probabilistic Cauchy theory \cite{OT}, and thus Theorem~\ref{THM:ASGWP} is sharp from this point of view.

 Initial data of the form \eqref{initialdata} is also natural as it corresponds to typical elements belonging to the support of the infinite-dimensional Gaussian measure $\mu_{\al}$ formally given by
\begin{align}
d\mu_{\al}=Z_{\al}^{-1}e^{ -\frac{1}{2}\|u\|_{H^{\al}(\T)}^{2}}du. \label{mual}
\end{align}
Using this perspective, we may rephrase Theorem~\ref{THM:ASGWP} as almost sure global well-posedness of BBM~\eqref{BBM0} with respect to the Gaussian measure $\mu_{\al}$ supported on $H^{\al-\frac 12-}(\T)$ for any $\al>\frac{1}{4}$. 

For the context of BBM~\eqref{BBM0}, the transport properties of Gaussian measures under the nonlinear flow of \eqref{BBM0} have been extensively studied~\cite{desuzzoni1, desuzzoni2, desuzzoni3, Tzv3}. As BBM is a Hamiltonian PDE with Hamiltonian given by the energy \eqref{energycons}, 
$\mu_{1}$ is an invariant measure and this was shown in~\cite{desuzzoni2}. 
For the Gaussian measures $\mu_{\al}$ with $\al\neq 1$, we no longer expect invariance. In this setting, it has been shown \cite{Tzv3, GTV} that, when $\al>2$, the push-forward of $\mu_{\al}$ under the solution map of \eqref{BBM0} is quasi-invariant (i.e.~mutually absolutely continuous) with respect to $\mu_{\al}$. In view of Theorem~\ref{THM:ASGWP}, it is natural to ask if this quasi-invariance property continues to hold for all $\al>\frac 14$. 
We also note that construction of solutions to nonlinear dispersive PDEs with Gaussian initial data of the form \eqref{initialdata} (and in higher dimensions) is heavily motivated by the study of invariant measures. 
We mention \cite{Bourgain1, CollOh, DNY2, DNY3} for a very small selection of related breakthroughs; see also the references therein and the surveys \cite{BOP4, DNYsurvey}.

 We now move to discussing the proof of Theorem~\ref{THM:ASGWP}.
 The main idea is to combine the $I$-method with the low-high method used by Bona-Tzvetkov~\cite{BonaTzvet} for the global well-posedness of BBM in $L^2$.  
 To explain how this combination works, let us first recall the arguments in \cite{FBBM} whose global theory was adapted from \cite{Leonardo, GKOT}, and specialised to the simpler model \eqref{BBM}.
 For fixed $\al\in (\tfrac 14, \tfrac 12]$, we consider the first order expansion:
 \begin{align*}
u(t) = u_0+v(t)
\end{align*}
 with $v$ satisfying 
 \begin{align}
\dt v(t) = -\tfrac{1}{2} \vp(\dx)( (u_0+v)^2) =   -\tfrac{1}{2} \vp(\dx)( v^2 +2u_0 v) -\tfrac{1}{2} \mathcal{N}(u_0), \label{veq}
\end{align}
with $v(0)=0$ and
where the quadratic object
\begin{align}
\mathcal{N}(u_0) : = \vp(\dx)(u_0^2)  \label{quadraticobject}
\end{align}
is constructed via probabilistic means and has the regularity $2\al-$. 
The construction of $\mathcal{N}(u_0)$ works without renormalisation as \eqref{BBM} is self-renormalising because the operator $\vp(\dx)$ vanishes on constants. 
Then, a contraction mapping argument can be closed to construct $v\in C([0,T_{\max});H^{s}(\T))$ for any $\frac 12 <s<2\al$ and for some (random) maximal existence time $T_{\max}=T_{\max}( u_0)>0$, depending only on norms of $u_0$ and $\mathcal{N}(u_0)$. Moreover, there is the blow-up criterion that $T_{\max}=+\infty$ unless
\begin{align}
\lim_{T\to T_{\max}} \|v(t)\|_{H^{s}(\T)}=\infty. \label{vhsbnd}
\end{align}
For almost sure global well-posedness, the goal in \cite{FBBM} was to refute \eqref{vhsbnd} when $\al=\frac 12$ by deriving an a priori bound on the remainder $v$.  Typically such bounds are based off  coercive conservation laws for which \eqref{BBM0} has only the energy
\begin{align}
E(u)&:=\tfrac 12 \|u\|_{H^{1}(\T)}^{2}. \label{energycons} 
\end{align}
There are then two difficulties:
(i) as soon as $\al\leq \frac 12$, $v$ does not have sufficient regularity for $E(v)$ to be finite, and (ii) $v$ solves the perturbed BBM equation  \eqref{veq} and so $E(v)$ would not be conserved even if $v$ were somehow smoother.

The first issue (i) motivated applying the $I$-method of Colliander-Keel-Staffilani-Takaoka-Tao~\cite{Iteam1, Iteam2} in this probabilistic context. Namely,
we smooth the initial data by applying the Fourier multiplier operator $I_{N}$ given by 
$\widehat{I_{N}f}(n)=m_{N}(n)\widehat{f}(n)$, $n\in\Z$, where $\,\widehat{\cdot}\,$ denotes the Fourier transform, and  $m_{N}(n)$ is the restriction to the integers of the smooth function $m:\R \to \R$ defined by 
\begin{align}
m_{N}(\xi):= m\bigg(\frac{\xi}{N}  \bigg) =
\begin{cases}
1    \,\,\, & \text{if}\,\,\, |\xi|\leq N,  \\
\Big( \frac{N}{|\xi|} \Big)^{1-s} \,\,\,&\text{if}\,\,\, |\xi|> N.
\end{cases} \label{Imult}
\end{align}
The operator $I_{N}$ is the identity on low frequencies and a fractional integral operator on high frequencies, hence the name $I$-method.
For simplicity of presentation, we will now drop the subscript $N$. 
It is easy to see that $Iv(t)\in H^{1}(\T)$ almost surely and satisfies 
\begin{align}
\begin{cases}
\partial_{t}Iv =  -\tfrac{1}{2}\vp(\dx) (  I(v^2)  +2 I(u_0 v)) -\tfrac12 I\mathcal{N}(u_0)   , \\
Iv|_{t = 0} = 0. 
\end{cases} 
\label{IBBM}
\end{align}
By defining the `modified energy' 
\begin{align}
\mathcal{E}(t) := E(Iv)(t)=\tfrac{1}{2}\| Iv(t)\|_{H^{1}(\T)}^{2}, \label{modenergy}
\end{align}
and observing $\|v(t)\|^2_{H^{s}(\T)} \les \EE(t)$
we reduce refuting \eqref{vhsbnd} to showing
\begin{align}
 \sup_{t\in [0,T]}\EE(t) \leq C(T)<\infty \label{modvbnd}
\end{align}
for any fixed $T>0$.
 Taking a time derivative of $\EE$ and inserting \eqref{IBBM}, we schematically arrive at 
\begin{align}
\frac{d}{dt}\EE(t) \approx  \int_{\T} (\dx Iv) I(v^2)dx+ \int_{\T}Iu_0 \, Iv \,\dx Iv \, dx + \text{lower order terms}. \label{schematicODE}
\end{align}
If the $I$ operator did commute with the product $v^2$, the first term on the right-hand side of \eqref{schematicODE}  would vanish due to the energy cancellation
\begin{align*}
\int_{\T} \dx Iv \cdot (Iv)^2 dx = 0
\end{align*}
 and we could close \eqref{schematicODE} with a simple bound 
\begin{align}
\frac{d}{dt}\EE(t)  \les \|I u_0\|_{L^2} \EE(t). 
\label{gronwall1}
\end{align}
As $u_0\not\in L^2$ a.s., we typically have
\begin{align}
\| I u_0 \|_{L^2} \sim \g_{N}(\al)^{\frac 12},
\end{align}
 where $\g_N(\al)$ was defined in \eqref{variance}.
Despite the growth in \eqref{variance}, \eqref{gronwall1} would be closable by a Gronwall argument. 
Of course, $I$ does not commute with products and this leads to a commutator which satisfies the bound
\begin{align}
\bigg|\int_{\T} (\dx Iv)[I(v^2)-(Iv)^{2}]\, dx \bigg| \les N^{-\be}\EE(t)^{\frac 32}, \label{commest}
\end{align}
for some $\be>0$ (in fact, $\be=\frac 32$).
In particular, we essentially end up with the Riccati-type ODE: 
\begin{align}
\frac{d}{dt}\EE(t) \les N^{-\be} \EE(t)^{\frac 32} +  \g_{N}(\al)^{\frac 12}\EE(t). \label{riccati}
\end{align}
As $\EE(0)=0$, this yields a bound on $\EE(t)$ for times $0\leq t \leq T_{N}$, as long as $T_{N}$ satisfies 
\begin{align}
T_{N}   \ll \frac{\be \log N}{\g_{N}(\al)^{\frac 12}}.  \label{blowupT}
\end{align}
In order to allow $T_{N}\to \infty$ as $N\to \infty$, \eqref{variance} forces us to take precisely $\al=\frac 12$, recovering the result in \cite{FBBM}. Two items become clear here: (i) it is impossible in this argument to take any $\al<\frac 12$, and, more subtly, (ii) there is \textit{no} benefit in optimising the value of $\be$ in \eqref{commest}; one only needs \eqref{commest} for \textit{some} $\be>0$. The last point (ii) is unlike the situation in the usual deterministic $I$-method, where optimising the values of $\be$ in the commutator estimates leads to improved regularity restrictions.

Our route to proving Theorem~\ref{THM:ASGWP} is to introduce a new modified argument that resolves both of the previous points: it goes beyond the logarithmic correlation barrier in \eqref{blowupT} and makes use of the value of $\be$ in \eqref{commest}. To this end, we combine the $I$-method argument above with the low-high method due to Bona-Tzvetkov~\cite{BonaTzvet}. 

 Let $T>0$ be the target time to construct a solution to \eqref{BBM1}. 
We then introduce a new parameter $M\in 2^{\N}$ and split the initial data at this frequency:
\begin{align*}
u_0 = \P_{\leq M} u_0 + \P_{>M}u_0.
\end{align*}
We let $v_{M}$ denote the solution to 
 \begin{align}
\dt v_{M} = -\tfrac{1}{2} \vp(\dx)( (\P_{>M}u_0+v_M)^2) =   -\tfrac{1}{2} \vp(\dx)( v_{M}^2 +2(\P_{>M}u_0) v_M) -\tfrac{1}{2} \mathcal{N}(\P_{>M}u_0), \label{veqM}
\end{align}
with $v_{M}(0)=0$. Given $\frac 12 -\al<\s<2\al$, we have 
\begin{align*}
\|  \P_{>M}u_0\|_{W^{-1+\s, \infty}} \les_{\o} M^{-\frac 12-\al+\s+} \|u_0\|_{W^{\al-\frac 12-,\infty}} \quad \text{and} \quad \| \mathcal{N}(\P_{>M}u_0)\|_{H^{\s}} \les_{\o} M^{\s-2\al+}
\end{align*}
almost surely, which are \textit{decaying} factors in $M$. Thus, by a probabilistic local-in-time theory,
we may choose $M=M(\o,T)$ sufficiently large, so that $v_{M}\in C([0,2T];H^{\s}(\T))$ solving \eqref{veqM} exists on the entire interval $[0,T]$ and moreover, satisfies 
\begin{align}
\| v_{M}\|_{C([0,2T]; H^{\s})} \les M^{\s-2\al+} \to 0 \label{vMlimit}
\end{align}
as $M\to \infty$.
Next, we consider $w_{M}$ which is the evolution from the low-frequency contribution of the initial data:
\begin{align}
\begin{cases}
\partial_{t}w_M =  -\tfrac{1}{2}\vp(\dx) (  w_{M}^2 +2(\P_{>M}u_0+v_M)w_M   )   , \\
w_M\vert_{t = 0} = \P_{\leq M}u_0. 
\end{cases} 
\label{BBMwM}
\end{align}
Then, $u= \P_{>M}u_0 + v_M +w_M$ solves the original BBM equation \eqref{BBM1} with data $u_0$. Even though the initial data for $w_M$ is smooth, the rough forcing term $\P_{>M}u_0$ limits the regularity of $w_M$ so that, for short times, we may only construct $w_M(t) \in H^{s}(\T)$ for $\frac 12 -\al<s<\al+\frac 12$. In the following, we will use the flexibility afforded by measuring $v_M$ and $w_M$ with the different regularities $\s$ and $s$, respectively (which gives rise to the minimum regularity in Theorem~\ref{THM:ASGWP}).

As $s<1$, to globalise $w_M$, and hence establish the almost sure global well-posedness for \eqref{BBM1}, we want to prove an a priori bound and for this we use the $I$-method approach detailed before. Using the same modified energy $\EE$ for $w_M$ as in \eqref{modenergy}, we essentially have
\begin{align}
\bigg\vert \frac{d}{dt}\EE(t)  \bigg\vert \les   N^{-\be} \EE(t)^{\frac 32} + \big(\| I \P_{>M}u_0\|_{L^2}+\|v_M(t)\|_{L^2})   \EE(t) . \label{schematicODE2}
\end{align}
The key difference compared to \eqref{schematicODE} is that the factors on the linear term $\EE(t)$ can be made \textit{small} in view of \eqref{vMlimit} and the estimate
\begin{align}
\| I \P_{>M}u_0\|_{L^2} \les_{\o} N^{1-s}M^{-\frac 12-\al+s}\|u_0\|_{H^{\al-\frac 12-}} \label{Ihigh}
\end{align}
which follows from enforcing $M>2N$.  
We then arrive at 
\begin{align}
\bigg\vert \frac{d}{dt}\EE(t)  \bigg\vert \les   N^{-\be} \EE(t)^{\frac 32} + [ N^{1-s}M^{-\frac 12-\al+s} +M^{-2\al+\s+}]  \EE(t),    \label{schematicODE3}
\end{align}
This significantly weakens \eqref{blowupT} at the expense of introducing the large initial data $\EE(0) = \| I \P_{\leq M}u_0\|_{H^{1}}^2$.
Nonetheless, by taking $M\sim N^{k}$ for some $k>1$, a Gronwall inequality gives a bound on $\EE(t)$ for all $t<2T$ as long as $M,N$ satisfy both of
\begin{align*}
2T [ N^{1-s}M^{-\frac 12-\al+s} +M^{-2\al+\s+}] \ll 1 \,\,\, \text{and} \,\,\, \EE(0)^{\frac 12} \ll N^{\be} [ N^{1-s}M^{-\frac 12-\al+s} +M^{-2\al+\s+}].
\end{align*}
See the proof of Proposition~\ref{prop:bdk} for details leading to these conditions.
The better decay from \eqref{Ihigh} and the value of $\be$ in \eqref{commest} conspire enough so that we can meet these conditions by choosing $M\sim N^{k(\al,s)}$ for some $k>1$ not too large, and for $s<3\al -\frac 12$ (and any $\s<2\al$). The restriction on $s$ is then consistent with the lower bound $s>\frac 12 -\al$ when $\al>\frac 14$, which completes the argument.

We point out that this argument is different to that of Colliander-Oh \cite{CollOh} for the one-dimensional cubic NLS with Gaussian random initial data as their argument is based on adapting, to the probabilistic setting, Bourgain's high-low argument \cite{Bourgain3}. Briefly, the initial data is also split into high and low parts. However, the evolution of the low part is solved exactly, while that for the high part is solved via the difference equation. The argument iterates based on a higher conservation for the low part (say, $L^2$) and nonlinear smoothing for the high part.

The method developed in this paper may be applicable to similar globalisation problems and the author intends to pursue these directions in future work.

  \begin{remark}\rm \label{RMK:BBM}
Using the translation \eqref{translate}, Theorem~\ref{THM:ASGWP} implies almost sure global well-posedness for \eqref{BBM0} with initial data of the form $u_{0}^{\o}-1$. Then, we obtain almost sure global well-posedness for \eqref{BBM0} with initial data exactly of the form $u_0^{\o}$ as a consequence of the Cameron-Martin Theorem; see \cite{OQ} for similar details in the context of the cubic nonlinear Schr\"{o}dinger equation.
 \end{remark}

 \begin{remark}\rm \label{remark: gaussianity}
 Theorem~\ref{THM:ASGWP} also extends to random initial data \eqref{initialdata} where the random variables $\{g_{n}\}_{n\in \Z}$ are not necessarily Gaussian. For precise assumptions, see those in \cite[Appendix A]{FBBM}. 
In fact, as the random object $\NN(u_0)$ is no longer time dependent, we only need second-moment estimates on $\|\NN(\P_{>M}u_0)\|_{H^{\s}}$ which is technically simpler than those in \cite[Appendix A]{FBBM}.
\end{remark}

\section{Deterministic and probabilistic tools}\label{sect:lemmas}

\subsection{Deterministic tools}

First, we recall the following key bilinear estimate, due to Bona-Tzvetkov \cite{BonaTzvet} (see also Roum\'egoux~\cite{Roumegoux}).

\begin{lemma}[\cite{BonaTzvet, Roumegoux}]\label{LEM:BT}
For any $s\geq 0$ and any $f,g\in H^{s}(\T)$, we have \begin{equation}
\| \varphi(  \dx )(fg)\|_{H^{s}(\T)} \lesssim \|f\|_{H^{s}(\T)}\|g\|_{H^{s}(\T)}. \label{productestimateBT}
\end{equation}
\end{lemma}

We also need the following paraproduct estimate:

\begin{lemma} [{\cite[Lemma 3.4]{GKO}}] \label{LEM:negprod} Let $0\leq s\leq 1$ and suppose that $1<p,q,r<\infty$ satisfy $\frac{1}{p}+\frac{1}{q}=\frac{1}{r}+s$. Then, we have 
\begin{align*}
\| \jb{\dx}^{-s}(fg)\|_{L^{r}(\T)}\lesssim \| \jb{\dx}^{-s}f\|_{L^{p}(\T)}\| \jb{\dx}^{s}g\|_{L^{q}(\T)}. 
\end{align*}
\end{lemma}

\subsection{Probabilistic tools}

We only need the following version of the Wiener chaose estimate for Wiener chaoses of order at most 2. See for instance \cite[Proposition 2.4]{ThomTzvet} for more details. 

\begin{lemma}[Wiener chaos estimate]\label{Lemma:wienerchaos}  
For any $2\leq p<\infty$ and $a\in \l^2(\Z)$ and $b\in \l^2(\Z^2)$, it holds that 
\begin{align}
 \bigg\| \sum_{n\in \Z}a(n)g_{n}(\o)  \bigg\|_{L^{p}(\Omega)} &\lesssim p^{\frac{1}{2}}\|a(n)\|_{\l^{2}_{n}}. \label{wienerchaos} \\
\bigg\|  \sum_{n_1,n_2} b(n_1,n_2) g_{n_1}g_{n_2}  \bigg\|_{L^{p}(\O)} &\les p  \bigg\|  \sum_{n_1,n_2} b(n_1,n_2) g_{n_1}g_{n_2}  \bigg\|_{L^{2}(\O)}
\end{align}
where the family $\{g_n\}_{n\in \Z}$ are as in \eqref{initialdata}.
\end{lemma}

We now study the regularity and integrability properties of the random initial data \eqref{initialdata} and the bilinear term $\NN(z)$ given in \eqref{quadraticobject}.

\begin{proposition}\label{prop:stochobjects} 
Let $\frac 14 < \al \leq \frac 12$,
\begin{align*}
s_1 <\al-\tfrac 12 \quad \text{and} \quad s_2<2\al.
\end{align*}
 Let $\{ \rho_k \}_{k\in \N}$ be a family of mollifiers on $\T$ and let $u_{0,k} =u_0^{\o}\ast \rho_k$, where $u_0^{\o}$ is as in \eqref{initialdata}. 
Then, 
\begin{align*}
(u_{0,k}, \NN(u_{0,k})) \longrightarrow (u_0, \NN(u_0)),
\end{align*}
as $k \to \infty$ in $L^{q}(\O;  W^{s_1,\infty}(\T)\times W^{s_2,\infty}(\T))$ for any $q\geq 1$ and almost surely in $ W^{s_1,\infty}(\T)\times W^{s_2,\infty}(\T)$. Moreover, the limit $(u_0,\NN(u_0))$ is independent of the choice of mollification kernel $\rho$, including the regularisation by the (non-smooth) Dirichlet projection $\P_{\leq N}$.
Furthermore, there exist $C,C'>0$ such that 
\begin{align}
\begin{split}
\prob \bigg( \|u_0\|_{H^{\al-\frac 12-}}+ \sup_{M\in 2^{\N}} &\big[ M^{\al-\frac 12-s_1-} \| \P_{> M}u_0\|_{W^{s_1,\infty}(\T)}  \\
&+  M^{2\al-s_2-} \|\NN( \P_{> M}u_0)\|_{H^{s_2}(\T)} \big]  >L \bigg) \leq C'e^{-CL}
\end{split}
\label{probM}
\end{align}
 for any $L >0$. 
\end{proposition}

\begin{proof}
The entire statement regarding convergence of mollified sequences and independence of the limits is proved in \cite[Proposition 2.6]{FBBM}. The only new part is the estimate on the event in \eqref{probM}. The statements for $\P_{> M}u_0$ follow from Sobolev embedding and the estimate 
\begin{align*}
\sum_{|n|>M} \frac{\jb{n}^{2s_1}}{\jb{n}^{2\al}} \les M^{2s_1-2\al+1-}.
\end{align*}
 For $\NN(\P_{> M}u_0)$, we provide details only for the second moment, as higher moments, and thence, the exponential decay follow from Lemma~\ref{Lemma:wienerchaos}. We have 
\begin{align*}
 \E \big[ \| \NN(\P_{> M}u_0)\|_{H^{s_2}}^{2}    \big] &= \sum_{n\neq 0} \jb{n}^{2s_2} |\vp(n)|^2 \E \bigg[ \bigg| \sum_{ \substack{n=n_1+n_2 \\ |n_1|, |n_2|> M}} \frac{g_{n_1}g_{n_2}}{\jb{n_1}^{\al} \jb{n_2}^{\al}} \bigg|^{2} \bigg] \\
 & \sim \sum_{n\neq 0} \jb{n}^{2s_2-2}\sum_{ \substack{n=n_1+n_2 \\ |n_1|, |n_2|> M}}  \frac{1}{\jb{n_1}^{2\al} \jb{n_2}^{2\al}},
\end{align*}
where in evaluating the expectation, we crucially used that fact that $n_1+n_2 \neq 0$ since  $\vp(0)=0$. We consider two sub-cases in the above summation (assuming by symmetry that $|n_1|\geq |n_2|$). If $|n_1|\sim |n_2|$, then since $\al>\frac 14$, we get 
\begin{align*}
\sum_{n\neq 0} \jb{n}^{2s_2-2}\sum_{ \substack{n=n_1+n_2 \\ |n_1|\sim |n_2|> M}}  \frac{1}{\jb{n_1}^{2\al} \jb{n_2}^{2\al}} \les \sum_{n} \frac{\jb{n}^{2s_2-2}}{ \max(|n|, M)^{4\al-1}}  \les M^{2s_2-4\al}.
\end{align*}
On the the other hand, if $|n_1|\gg |n_2|$, then we have 
\begin{align*}
\sum_{n\neq 0} \jb{n}^{2s_2-2}\sum_{ \substack{n=n_1+n_2 \\ |n_1|\gg |n_2|> M}}  \frac{1}{\jb{n_1}^{2\al} \jb{n_2}^{2\al}} &\les \sum_{|n|\ges M} \jb{n}^{2s_2-2-2\al} \sum_{|n_2|\ll |n|} \frac{1}{\jb{n_2}^{2\al}}  \\ 
&\les  \sum_{|n|\ges M} \jb{n}^{2s_2-1-4\al} \les M^{2s_2 -4\al}.
\end{align*}
Then, in order to obtain the supremum over $M$, we use that for $M\in 2^{\N}$, it is enough to convert the supremum to a summation at the cost of a logarithm which explains the worse rates in \eqref{probM}.
This completes the proof.
\end{proof}

\section{Modified probabilistic local theory }\label{section:localt}

For the rest of the article, we restrict ourselves to the remaining regime $\frac 14<\al \leq \frac 12$ not handled in \cite{BonaTzvet}.
We first provide an abstract well-posedness for a perturbed BBM equation. In order to reach any target time $T>0$, we enforce that the perturbations themselves can be chosen sufficiently small.

\begin{proposition}\label{prop:detHIGH} 
Fix $T>0$, $0<\s<1$, $L\geq 1$, $a_1>0$ and $0<a_3<a_2$. Then, there exists $M_0=M_0(L,T,a_1,a_2,a_3)\geq 1$ such that for every $M\geq M_0$ the following holds:
For every pair of forcings $(z_1,z_2)$ satisfying 
 \begin{align}
\|z_{1}\|_{W^{-s_1,\frac{1}{s_1}}( \T)} \leq L M^{-a_1}    \quad \text{and} \quad \|z_{2}\|_{H^{\s}( \T)}\leq LM^{-a_2}, \label{Kcond}
\end{align} 
where $s_1:= \min(1-\s,\s)$, 
 there exists a unique solution $v\in C([0,T]; H^{\s}(\T))$ to the initial value problem:
 \begin{align}
\begin{cases}
\partial_{t}v =  -\tfrac{1}{2}\vp(\dx) (  v^2 +2z_1 v  )  -\tfrac{1}{2}z_2 , \\
v\vert_{t = 0} =0,
\end{cases} 
\label{BBMz1z2}
\end{align}
which satisfies 
\begin{align}
\| v\|_{C([0,T];H^{\s})} \leq M^{-a_3}.
\end{align}
\end{proposition}

\begin{proof}
We apply the contraction mapping argument to the integral formulation:
\begin{align*}
v(t) =\G_{t}[v]:= -\frac 12 \int_{0}^{t}   \vp(\dx) (  v^2 +2z_1 v  ) dt' + tz_2
\end{align*}
in the small ball 
\begin{align*}
B_{R} : = \{  v\in C([0,T];H^{\s}(\T)) \, : \, \|v\|_{C_{T}H^{\s}} \leq R\}
\end{align*}
where $R=M^{-a_3}$.
The main term to consider is the linear in $v$ term $\vp(\dx)[ z_1 v]$. When $0<\s\leq \frac 12$,  Lemma~\ref{LEM:negprod} implies
\begin{align}
\| \vp(\dx)[ z_1 v]\|_{H^{\s}} \les \| \jb{\dx}^{-(1-\s)}[ z_1 v]\|_{L^2} \les \| \jb{\dx}^{-\s}[ z_1 v]\|_{L^2} \les \| z_1\|_{W^{-\s, \frac{1}{\s}}} \| v\|_{H^{\s}}. \label{zest1}
\end{align}
Similarly, if $\frac 12<\s<1$, then Lemma~\ref{LEM:negprod} also yields
\begin{align}
\| \vp(\dx)[ z_1 v]\|_{H^{\s}} \les \| \jb{\dx}^{-(1-\s)}[ z_1 v]\|_{L^2} \les  \| z_1\|_{W^{-(1-\s), \frac{1}{1-\s}}} \| v\|_{H^{1-\s}} \les \| z_1\|_{W^{-(1-\s), \frac{1}{1-\s}}} \| v\|_{H^{\s}}. \label{z2est}
\end{align}
Applying these estimates and using Lemma~\ref{LEM:BT} and \eqref{Kcond}, we then find for $v\in B_{R}$ and $M\geq M_0$,
\begin{align*}
\| \G_{t}[v]&\|_{C_{T}H^{\s}} \leq CT \|v\|_{C_{T}H^{\s}}^{2}  +CT\|z_1\|_{W^{-s_1,\frac{1}{s_1}}}  \|v\|_{C_{T}H^{\s}} + CT \|z_2\|_{H^{\s}} \\
& \leq CTR^2 + CT LM^{-a_1}R + CTLM^{-a_2} \leq [CTM_0^{-a_3} + CT LM_0^{-a_1} + CTLM^{-a_2+a_3}]R.
\end{align*}
By choosing $M_0=M_0(T,L,a_1,a_2,a_3)\geq 1$ such that
\begin{align}
M_0 \geq (1+ 3CTL)^{\max( \frac{1}{a_3}, \frac{1}{a_1}, \frac{1}{a_2-a_3})}, \label{Mcond}
\end{align}
we see that $\G_{t}[v]:B_{R} \to B_{R}$. The condition \eqref{Mcond} also ensures that $\G_{t}[v]$ is a strict contraction on $B_{R}$, which then completes the proof by invoking the contraction mapping theorem.
\end{proof}

We now apply Proposition~\ref{prop:detHIGH} to the case when $z_1 = \P_{>M}u_0$ and $z_2=\NN(\P_{>M}u_0)$.

\begin{proposition}[Global well-posedness for $v_{M}$]\label{PROP:randomHIGH} 
Fix $T>0$, $\frac 14<\al \leq \frac 12$, $\frac 12-\al<\s<2\al$, and $L\geq 1$. Then, there exists an event $\Sigma_{L,\s}\subset \O$ satisfying $\PP(\Sigma_{L,\s}^{c}) \leq C'e^{-CL}$, such that for every $\o\in \Sigma_{L,\s}$,  there exists $M_0=M_0(L,T,\al,\s)\geq 1$ such that for every $M\geq M_0$, there is a unique solution $v_{M}\in C([0,2T];H^{\s}(\T))$ to the initial value problem:
 \begin{align}
\begin{cases}
\partial_{t}v_{M} =  -\tfrac{1}{2}\vp(\dx) (  v_{M}^2 +2(\P_{> M}u^{\o}_0) v_M  )  -\tfrac{1}{2}\NN(\P_{>M}u^{\o}_0) , \\
v_M\vert_{t = 0} =0,
\end{cases} 
\label{BBMvM}
\end{align}
which satisfies 
\begin{align}
\| v_{M}\|_{C([0,2T];H^{\s}(\T))} \leq M^{-2\al+\s+}, \label{vMdecay}
\end{align}
where $u_0^{\o}$ is as in \eqref{initialdata} and $\NN(\P_{>M}u_0^{\o})$ was constructed in Proposition~\ref{prop:stochobjects}.
\end{proposition}
\begin{proof}
Let 
\begin{align*}
\Sigma_{L,\s} : =  \bigg\{\o\in \O\,:\,  \sup_{M\in 2^{\N}} \big[ M^{\al-\frac 12-s_1-} \| \P_{> M}u^{\o}_0\|_{W^{s_1,\infty}(\T)} +  M^{2\al-\s-} \|\NN( \P_{> M}u^{\o}_0)\|_{H^{\s}(\T)} \big]  \leq L \bigg\}
\end{align*}
with $s_1 : = -\min(\s,1-\s)$.  
In order to apply Proposition~\ref{prop:stochobjects}, we need to check that $s_1 <\al-\frac 12$. 
If $\s<\frac 12$, then this imposes $\s>\frac 12-\al$, while if $\s\geq \frac 12$, we find $\s<\al+\frac 12$. Thus, \eqref{probM} yields that $\PP(\Sigma_{L,\s}^{c})\leq C' e^{-CL}$.
For  $\o \in \Sigma_{L,\s}$, we then
apply Proposition~\ref{prop:detHIGH}  with $z_1 =\P_{>M}u_0$ and $z_2 = \NN( \P_{> M}u_0)$ so that $a_1= \al -\frac 12-s_1->0$ (as we just checked that $s_1<\al-\frac 12$) and $a_2= 2\al-\s-$. 
\end{proof}

We now  provide a basic local well-posedness theory for the difference equation
\eqref{BBMwM}. As the variable $w_{M}$ carries initial data $\P_{\leq M}u_0$, which is \textit{not} small, the restriction to short times is necessary. The main goal in the next section is to establish an a priori bound on $w_{M}$ which allows to iterate the local theory for $w_{M}$ to cover the full time interval $[0,T]$.

\begin{lemma}[Local theory for $w_{M}$] \label{LEM:wM}
Let $T>0$, $\al, \s>0$ and the event $\Sigma_{L,\s}$ be as in Proposition~\ref{PROP:randomHIGH}. 
Then, for $\frac 12-\al<s<\al+\frac 12$, there exists another event $\Sigma^{(2)}_{L,T,s}\subset \Omega$ satisfying $\PP(( \Sigma^{(2)}_{L,T,s})^{c}) \leq C'e^{-CL^{2}}$ 
for each $\o\in \Sigma^{(3)}_{L,\s,s,T}:= \Sigma_{L,\s}\cap \Sigma^{(2)}_{L,s,T}$, and for each $M\geq M_0$, with $v_{M}$ as constructed in Proposition~\ref{PROP:randomHIGH} on $[0,2T]$,  there exists a random maximal time $0<T_{\max}^{\o}<T$ for which there exists a unique solution $w_{M}\in C([0,T_{\max}^{\o}); H^{s}(\T))$ to \eqref{BBMwM} with $w_{M}\vert_{t=0}=\P_{\leq M}u_0^{\o}$. 
\end{lemma}

\begin{proof}
The main point is that to handle a different regularity parameter $s$, we add in the event 
\begin{align*}
\Sigma^{(2)}_{L,T,s} =  \bigg\{\o\in \O\,:\,  \sup_{M\in 2^{\N}} \big[ M^{\al-\frac 12-s_3-} \| \P_{> M}u^{\o}_0\|_{W^{s_3,\infty}(\T)} \leq L \bigg\}
\end{align*}
with $s_3 := -\min(1-s,s)$. Then, since $\frac 12-\al<s<\al+\frac 12$, Proposition~\ref{prop:stochobjects} applies to show that the complement of $\Sigma^{(2)}_{L,T,s}$  probability exponentially small in $L$. Then, on the event $ \Sigma^{(3)}_{L,\s,s,T}$, we set up and solve the integral formulation of \eqref{BBMwM} on a small interval $[0,\tau]$ for some $0<\tau <T$. More precisely, we consider the integral formulation
\begin{align*}
w_{M} (t) = \P_{\leq M}u_0 -\frac{1}{2}\int_{0}^{t} \vp(\dx)\big[ w_{M}^2 +2(\P_{>M}u_0+v_M)w_{M}\big]dt'.
\end{align*}
and run a contraction mapping argument in $C([0,\tau];H^{s}_{x}(\T))$. Using Bernstein's inequality, we have
\begin{align*}
\|\P_{\leq M}u_0\|_{H^{s}} \les M^{s+\frac 12-\al+} \|u_0\|_{H^{\al-\frac 12-}} \les M^{s_3+\frac 12-\al+}L.
\end{align*}
To handle the multiplication with $\P_{>M}u_0$, we use \eqref{zest1} and \eqref{z2est}. As for the multiplication by $v_{M}$, if $s<\frac 12$, we have 
\begin{align*}
\| \vp(\dx) [ v_M w_M] \|_{H^s_x} \les \|\jb{\dx}^{-\frac 12-}[v_{M}w_M]\|_{L^2} \les \|v_{M}\|_{L^2}\|w_{M}\|_{L^2}.
\end{align*}
If instead $s\geq \frac 12$, then we have 
\begin{align*}
\| \vp(\dx) [ v_M w_M] \|_{H^s_x}  \les \|v_{M}w_{M}\|_{L^{2-}_x} \les \|v_{M}\|_{L^2_x} \|w_{M}\|_{L^{\infty-}_x} \les \|v_{M}\|_{L^2_x} \|w_M\|_{H^{s}_x}.
\end{align*}
The improved upper bound $s<\al +\frac 12$ is due to the absence of the quadratic object $\NN(u_0)$ in \eqref{BBMwM}. Note that $\tau$, and hence, $T_{\max}^{\o}$, depend upon $M$.
\end{proof} 

\section{Probabilistic global theory on $\T$}\label{section:globalt}

The goal of this section is to establish that for some $M\in 2^{\N}$, we have an a-priori bound on  the growth  of $w_{M}$ with high-probability. More concretely, we aim to prove the following.

\begin{proposition}\label{prop:bdk}
 Let $\frac 14<\al\leq \frac 12$ and $\frac 12-\al<s<3\al-\frac 12$. Given $T, \e>0$, there exist $\tilde{\O}_{T,\e}\subset \O$ such that 
\begin{align*}
\prob( (\tilde{\O}_{T,\e})^{c}) <\e,
\end{align*}
 a large $M_1 (T,\eps) \geq M_0$, and a finite constant $C(T,\e)>0$ such that the following bound holds: 
\begin{align}
 \sup_{t\in [0,T]}\|w_{M_1}^{\o}(t)\|_{H^{s}(\T)} \leq C(T,\e), \label{bdvk}
\end{align}
for every solution $w_{M_1}^{\o}$ to \eqref{BBMwM} with $\o \in \tilde{\O}_{T,\eps}$.
\end{proposition}

Note that strictly speaking  for the arguments in this section to be completely rigorous one needs to mollify the random initial data $u_0$ and consider mollified versions of $v_M$ and $w_{M}$ to justify differentiations and using the equations pointwise. In order to not obscure the main ideas further, we omit such arguments and formally derive an a-priori bound for $w_{M}$. For further details on such a rigorous procedure, see \cite[Section 5]{FBBM}.

To obtain the bound \eqref{bdvk}, we will apply the $I$-method in this probabilistic context, which we now describe. 
Given $N\geq 1$, let $I_{N}=I$ be the Fourier multiplier operator defined by $\widehat{If}(n)=m_{N}(n)\widehat{f}(n),$ where $m_{N}$ is defined in \eqref{Imult}. The operator $I$ is smoothing of order $(1-s)$ and satisfies 
\begin{equation}
\|f\|_{H^{s}}\lesssim \|If\|_{H^{1}} \les N^{1-s}\|f\|_{H^{s}}. \label{hstoIh1}
\end{equation}
Thus, to obtain \eqref{bdvk}, it suffices to obtain a bound on $Iw_{M}$ in $H^{1}$ on the whole interval $[0,T]$.

We first control the random initial data.
\begin{lemma}\label{lemma:izintmoment}  
Let $\frac 14<\al\leq \frac 12$, $M,N \in \N$ be such that $M>2N$, and $s<\al+\frac 12$. Then, 
\begin{align}
\| I\P_{>M} u_0 \|_{L^2}  &\les N^{1-s} M^{-\frac 12 -\al+s+} \|u_0\|_{H^{\al-\frac 12-}}, \label{>Mdata} \\
\| I\P_{\leq M} u_0 \|_{H^1}& \les \big(  N^{\frac 32 -\al+}+ N^{1-s}M^{s+\frac 12-\al+}  \big)\|u_0\|_{H^{\al-\frac 12-}}. \label{<Mdata}
\end{align} 
\end{lemma}

\begin{proof} 
These follow from \eqref{Imult}. Indeed, as $M>2N$, 
\begin{align*}
\| I\P_{>M} u_0 \|_{L^2} \les N^{1-s} \| \jb{\dx}^{-1+s}\jb{\dx}^{\frac 12-\al+} \P_{>M}u_0\|_{H^{\al-\frac 12-}} \les N^{1-s} M^{-\frac 12 -\al+s+} \|u_0\|_{H^{\al-\frac 12-}}.
\end{align*}
Similarly, for \eqref{<Mdata}, we split $u_0 = \P_{\leq N}u_0 + \P_{>N}u_0$ and use \eqref{hstoIh1}:
\begin{align*}
\| I\P_{\leq M} u_0 \|_{H^1}& \leq \| I\P_{\leq N} u_0 \|_{H^1} + \| I \P_{>N} \P_{\leq M} u_0 \|_{H^1} \\
& \les  N^{\frac 32-\al+} \|u_0\|_{H^{\al-\frac 12-}} + N^{1-s} \| \jb{\dx}^{s}\P_{\leq M}u_0\|_{L^2} \\
& \les N^{\frac 32-\al+} \|u_0\|_{H^{\al-\frac 12-}}+N^{1-s}M^{s+\frac 12-\al+} 
\|u_0\|_{H^{\al-\frac 12-}}. \qedhere
\end{align*} 
\end{proof}

Note that it is possible to handle the bounds \eqref{>Mdata} and \eqref{<Mdata} probabilistically, which also makes visible the logarithmic dependence if $\al=\frac 12$ and removes the $\eps$-losses otherwise. As the $\eps$-losses are harmless, we preferred the simpler pathwise argument above.

Applying the $I$-operator to \eqref{BBMwM} we see that $Iw_{M}$ satisfies 
\begin{align}
\begin{cases}
\partial_{t}Iw_{M} = -\tfrac{1}{2} \vp(\dx)\left[  I(w_M^{2})+2 I \big(  (\P_{>M}u_0 +v_M)w_M\big)    \right]\\
Iw_{M}|_{t = 0} = I \P_{\leq M}u_0.
\end{cases} 
\label{IwM}
\end{align}
We define the modified energy functional $E(Iw_{M})(t):=\frac{1}{2}\|Iw_{M}(t)\|_{H^1}^{2}$. Using \eqref{IwM}, we compute 
\begin{align*}
E(Iw_M)(t)-E(Iw_M)(0) 
& = \frac{1}{2}\int_{0}^{t}\int_{\T} (\partial_{x}Iw_M)[I(w_M^{2})-(Iw_{M})^{2}]\,dxdt' \tag{I} \\
& \,\,\,\,\,\,\,\,+\int_{0}^{t}\int_{\T} (\partial_{x}Iw_M)I\big(  (\P_{>M}u_0 +v_M)w_M\big)\,dxdt' \tag{II}.
\end{align*}
We now estimate (I) and (II), beginning with the main commutator estimate (I). Relative to \cite[Lemma 4.3]{FBBM}, we weaken the regularity assumption from $s>\frac 12$ to $s>0$ and obtain the exact exponent $\frac 32$ instead of $\frac 32-$. Note that if we left the restriction $s>\frac 12$, we would eventually need $3\al-\frac 12>\frac 12$, giving just $\al>\frac 13$.

\begin{lemma}\label{LEM:comm}
 Let $0\leq s\leq 1$ and $w\in H^{s}(\T)$. Then, we have
 \begin{align*}
\bigg \vert\int_{\T} (\partial_{x}Iw)[I(w^{2})-(Iw)^{2}]\,dx \bigg\vert\lesssim N^{-\frac{3}{2}}\|Iw\|_{H^{1}(\T)}^{3}.
\end{align*}
\end{lemma}
\begin{proof}
We recall the setup from the proof of \cite[Lemma 4.3]{FBBM}.
By Plancherel, we have 
\begin{multline*}
 \int_{\T} (\partial_{x}Iw)[I(w^{2})-(Iw)^{2}]\,dx  =\sum_{n_{1}+n_{2}+n_{3}=0}in_{3}m(n_{3})(m(n_{1}+n_{2})-m(n_{1})m(n_{2})) \prod_{j=1}^{3} \ft w(n_{j}).
\end{multline*}
We symmetrise the right-hand side to obtain $$ \sum_{n_{1}+n_{2}+n_{3}=0} M(n_{1},n_{2},n_{3}) \ft w (n_{1})\ft w (n_{2})\ft w(n_{3}),$$ where $M$ is the symmetric multiplier 
\begin{align*}
 M(n_{1},n_{2},n_{3})=\frac{i}{3} [ & n_{1}m(n_{1})(m(n_{2}+n_{3})-m(n_{2})m(n_{3})) \\ &+n_{2}m(n_{2})(m(n_{1}+n_{3})-m(n_{1})m(n_{3}))\\ &+n_{3}m(n_{3})(m(n_{1}+n_{2})-m(n_{1})m(n_{2}))].
\end{align*} 
By symmetry, we assume $|n_{3}|\leq |n_{2}|\leq |n_{1}|$. Furthermore, we assume $|n_{1}|>N$, since otherwise $m(n_{j})=1$ for all $j=1,2,3$, which implies $M(n_{1},n_{2},n_{3})=0$ on $n_{1}+n_{2}+n_{3}=0$. In addition, we also assume $|n_{2}|\gtrsim N$ since if $|n_{2}|\ll N$ we obtain a contradiction to the conditions $n_{1}+n_{2}+n_{3}=0$, $|n_{1}|>N$ and $|n_{3}|\leq |n_{2}|$. For shorthand, let us define 
\begin{multline*} 
\Lambda_{N}:=\{ (n_{1},n_{2},n_{3})\in \Z^{3}  :  n_{1}+n_{2}+n_{3}=0,\, |n_{3}|\leq |n_{2}|\leq |n_{1}|,\, |n_{1}|>N, \, |n_{2}|\gtrsim N \,\}.
\end{multline*}
Using the condition $n_{1}+n_{2}+n_{3}=0$, we see that
$$ M(n_{1},n_{2},n_{3})=\frac{i}{3}\left[n_{1}m^{2}(n_{1})+n_{2}m^{2}(n_{2})+n_{3}m^{2}(n_{3}) \right],$$ and hence on $\Lambda_{N}$, we have
\begin{equation}
|M(n_{1},n_{2},n_{3})|\lesssim |n_{3}|m^{2}(n_{3}). \label{Mbd}
\end{equation} 
Setting $ y(n)=\jb{n}m(n)|\ft w(n)|$ and using \eqref{Mbd}, we have thus reduced to showing 
\begin{equation} \label{v2com1}
\sum_{\Lambda_{N}}\frac{ m(n_3) }{\jb{n_{1}}\jb{n_{2}} m(n_{1}) m(n_{2})} \prod_{j=1}^{3} y(n_j) \lesssim N^{-\frac{3}{2}}\|y(n)\|_{\l_{n}^{2}}^{3}.
\end{equation} 
At this point, we diverge from the analysis in \cite[Lemma 4.3]{FBBM}.
As $|n_1|\geq |n_2| \ges N$, we have 
\begin{align*}
\text{LHS} \eqref{v2com1} \les N^{-2(1-s)} \sum_{\Ld_{N}}  \frac{ |m(n_3)|}{\jb{n_1}^{s}\jb{n_2}^{s}} \prod_{j=1}^{3} y(n_j).
\end{align*}
We split the above summation into two cases. First, if $|n_3|<N$, then, by Cauchy-Schwarz, this contribution is bounded by
\begin{align*}
N^{-2(1-s)} \sum_{ |n_3|<N} y(n_3)  \sum_{|n_2|>N} \frac{y(n_2)y(-n_2-n_3)}{\jb{n_2}^{2s}} \les N^{-2} \bigg(  \sum_{ |n_3|<N} y(n_3) \bigg) \|y\|_{\l^2}^{2} \les N^{-\frac 32} \|y\|_{\l^2}^{3}.
\end{align*}
Next, if $|n_3|>N$, this contribution is bounded by 
\begin{align*}
N^{-1+s} \sum_{\substack{|n_3|>N\\ \Ld_{N}}} \frac{1}{ \jb{n_1}^{s}\jb{n_2}^{s} \jb{n_3}^{1-s}} \prod_{j=1}^{3} y(n_j) \les N^{-1+s} \sum_{|n_3|>N} \frac{y(n_3)}{\jb{n_3}^{1+s}} \| y\|_{\l^2}^{2} \les N^{-\frac 32} \|y\|_{\l^2}^{3},
\end{align*}
which completes the proof.
\end{proof}

Next, we give an estimate to handle the products $I( \P_{>M}u_0  \cdot w_{M})$ and $I(v_{M} w_M)$.

\begin{lemma}\label{LEM:comm2} 
Let $0\leq \g <\frac 12$ and $\g\leq s \leq 1-\g$. Then, 
\begin{equation}
\| I( fg)\|_{L^2} \les   \big( \|If \|_{L^2} + N^{-\frac 12+\g} \|f\|_{H^{-\g}}\big) \|I g\|_{H^{1}}  , \label{vzcom}
\end{equation}
\end{lemma}

\begin{proof}
By Plancherel and Cauchy-Schwarz, we have 
\begin{align*}
\| I( fg)\|_{L^2} &= \bigg\| \sum_{k}  \frac{m(n)}{m(k)\jb{k}} \ft f(n-k)  \jb{k}\ft{ Ig}(k) \bigg\|_{\l^2_n} \\
& \les  \bigg( \sum_{n,k}  \frac{m(n)^2}{m(k)^2\jb{k}^2}| \ft f(n-k)|^2 \bigg)^{\frac 12} \|Ig\|_{H^1} = :M_{N}(f)^{\frac 12} \|Ig\|_{H^1}.
\end{align*}
It remains to show that, under the conditions on $(\g,s)$ in the statement, that
\begin{align}
M_{N}(f) \les \|I f\|_{L^2}^2 + (N^{-\frac 12+\g} \|f\|_{H^{-\g}})^2. 
\label{MNest}
\end{align}
We consider a number of cases.

Suppose first that $|k|\ll |n|$. In this case, $|n|\sim |n-k|$ and so $m(n)^{2} \les m(n-k)^2$.
If additionally $|k|\les N$, then we have 
\begin{align*}
\sum_{|k|\les N } \frac{1}{m(k)^{2} \jb{k}^2} \sum_{n} | \ft{ I f}(n-k)|^{2} \les  \|If\|_{L^2}^2,
\end{align*}
uniformly in $N$. If instead $|k|\gg N$, then we have 
\begin{align*}
\sum_{ N\ll |k| \ll |n|}  \frac{m(n)^2}{m(k)^2\jb{k}^2}| \ft f(n-k)|^2 & \les \sum_{ N\ll |k| \ll |n|} \frac{ \jb{n-k}^{2\g}}{\jb{n}^{2-2s} \jb{k}^{2s}}| \jb{n-k}^{-\g}\ft f(n-k)|^2  \\
& \les \sum_{|k|\gg N} \frac{1}{ \jb{k}^{2-2\g}} \| \jb{n}^{-\g} \ft f (n)  \|_{\l^2_n} \les N^{-1+2\g} \|f\|_{H^{-\g}},
\end{align*}
where in the second inequality we used that $\g \leq 1-s$ and in the last inequality, we used $\g<\frac 12$.

Suppose now that $|k|\sim |n|$. Then, $ (\frac{m(n)}{m(k)\jb{k}})^2  \les \jb{k}^{-2}$ which is $\l^1$-summable. If $|n|\les N$, then we use $|k-n|\les N$ so $m(n-k)\sim 1$ which can be added to $f$ for free. Otherwise, $|n| \gg N$ and we bound instead by 
\begin{align*}
\sum_{|k|\gg N} \frac{1}{ \jb{k}^{2-2\g}} \sum_{n} | \jb{n-k}^{-\g}\ft f(n-k)|^2 \les N^{-1+2\g} \|f\|_{H^{-\g}}
\end{align*}
 for $\g<\frac 12$.
 
 Lastly, we suppose that $|k|\gg |n|$ and so $|k|\sim |n-k|$. If $|n|\les N$, then we have the multiplier 
 \begin{align*}
\frac{\jb{n-k}^{2\g}}{ m(k)^{2} \jb{k}^{2}} \les N^{-2(1-s)} \frac{\jb{k}^{2\g}}{\jb{k}^{2s}} \les N^{-2(1-s)} \frac{1}{\jb{k}^{2s-2\g}} \les N^{-2(1-s)} \frac{1}{\jb{n}^{2s-2\g}} 
\end{align*}
where we used $s\geq \g$. Then, we sum $|\jb{n-k}^{-\g}\ft f(n-k)|^2$ over $k$ and are left with 
\begin{align*}
\sum_{|n|\les N} N^{-2(1-s)} \frac{1}{\jb{n}^{2s-2\g}} \les N^{-1+2\g}.
\end{align*}
Lastly, if $|n|\gg N$, then 
\begin{align*}
\frac{ m(n)^2 \jb{n-k}^{2\g} }{ m(k)^{2} \jb{k}^{2}} \les \frac{1}{\jb{n}^{2-2s} \jb{k}^{2s-2\g} } \les \frac{1}{\jb{n}^{2-2\g} } 
\end{align*}
and summing over $|n|\gg N$ provides a factor $N^{-1+2\g}$. 
This completes the proof of \eqref{MNest}.
\end{proof}

\begin{proof}[Proof of Proposition~\ref{prop:bdk}] Fix $T,\e>0$.
Let $L>0$,  $\frac 14<\al \leq \frac 12$ and let $(\s,s)$ be as in Proposition~\ref{PROP:randomHIGH} and Lemma~\ref{LEM:wM}. We define
\begin{align*}
\O_{L,T,\al}& = \Sigma_{L,\s,s,T}^{(3)}\cap  \{  \|u_0\|_{H^{\al-\frac 12-}} \leq L \}.
\end{align*}
Then, Proposition~\ref{prop:stochobjects} and Lemma~\ref{LEM:wM} imply
\begin{align}
\prob( \O_{L,T,\al,\e}^{c}) < C'e^{-CL }. \label{smallset}
\end{align}
We fix $\o \in \O_{L,T,\al,\e}$ which, from Lemma~\ref{LEM:wM}, guarantees the existence of $M_0\in 2^{\N}$ such that for all $M\geq M_0$ dyadic, we have $v_{M} \in C([0,2T]; H^{\s}(\T))$ solving  \eqref{BBMvM} on $[0,2T]$ and satisfying \eqref{vMdecay}  and $w_{M} \in C([0,T_{\max}); H^{s}(\T))$ solving \eqref{BBMwM} on some random almost surely positive time interval $[0,T_{\max})$.   

For $N=N(L,T)>0$ to be determined later, we fix $M> \max(2N, M_0)$ also to be determined later.
Combining Lemma~\ref{LEM:comm}, Lemma~\ref{LEM:comm2} with $\g=\frac 12-\al-$ which needs $\frac 12 -\al < s<\frac 12 +\al$ and $\g=0$ (for $v_M$), \eqref{>Mdata}, and \eqref{vMdecay}, we have shown the following estimate:
\begin{align}
E(Iw_{M})(t) \leq & E(Iw_{M})(0)+C_{s} N^{-\frac 32}\int_{0}^{t}E(Iw_{M})^{\frac{3}{2}}(t')dt' \notag \\
& +  C_{s}\big[   M^{-2\al+\s+} + ( N^{1-s}M^{-\frac 12-\al+s+}+N^{-\al-}) \|u_0\|_{H^{\al-\frac 12-}}   \big] \int_{0}^{t}   E(Iw_{M})(t')dt'. \notag\\
 \leq & E(Iw_{M})(0)+C_{s} N^{-\frac 32}\int_{0}^{t}E(Iw_{M})^{\frac{3}{2}}(t')dt'  \notag \\
& +  C_{s}\big[   M^{-2\al+\s+} + ( N^{1-s}M^{-\frac 12-\al+s+}+N^{-\al-}) L   \big] \int_{0}^{t}   E(Iw_{M})(t')dt'.  \label{energy}
\end{align}
To simplify the notation, we define 
\begin{align}
A : =  N^{-\frac 32} \quad \text{and} \quad B: = M^{-2\al+\s+} + ( N^{1-s}M^{-\frac 12-\al+s+}+N^{-\al-}) L  . \label{AB}
\end{align}
 Let
\begin{align*}
\cj{T} =\sup \{ t\in [0,2T]\, :\, E(Iw_M)^{\frac 12}(t) \leq 4BA^{-1}  \},
\end{align*}
By enforcing that 
\begin{align}
 E(Iw_M)^{\frac 12}(0 ) \leq BA^{-1}, \label{cond1}
\end{align}
and using the continuity in time of $t\mapsto E(Iw_{M})(t)$, we see that $\cj{T}>0$. 

We claim that $\cj{T}\geq T$. If otherwise, we have that $\cj{T}<T$, then for all $t\in [0,\cj{T}]$, \eqref{energy} implies 
\begin{align*}
E(Iw_{M})(t) \leq E(Iw_{M})(0) + 5C_{s}B \int_{0}^{t} E(Iw_{M})(t')dt'.
\end{align*}
By Gronwall's inequality, we then obtain
\begin{align*}
E(Iw_{M})^{\frac 12}(t) \leq E(Iw_{M})^{\frac 12}(0)  e^{\frac 52 C_{s} Bt}.
\end{align*}
Continuity again implies that $E(Iw_{M})^{\frac 12}(\cj{T}) =4BA^{-1}$, and hence
\begin{align}
4BA^{-1} =E(Iw_{M})^{\frac 12}(\cj{T})\leq E(Iw_{M})^{\frac 12}(0)  e^{\frac 52 C_{s} B\cj{T}} \leq E(Iw_{M})^{\frac 12}(0)  e^{\frac 52 C_{s} BT}. \label{contradiction}
\end{align}
Enforcing the condition that
 \begin{align}
\tfrac{5}{2}C_{s} BT \leq \log 2 \label{cond2}
\end{align}
it then follows from \eqref{cond1} that \eqref{contradiction} yields a contradiction. We thus conclude that $\cj{T}\geq T$, \textit{provided} that both \eqref{cond1} and \eqref{cond2} hold true, which we verify now.
As $E(Iw_M)^{\frac 12}(0 )= \| I \P_{\leq M}u_0\|_{H^1}$,  \eqref{>Mdata},  \eqref{<Mdata}, \eqref{AB} show that \eqref{cond1} and \eqref{cond2} are implied by
\begin{align}
 \text{(i)}& \quad N^{\frac 32-\al+} +N^{1-s}M^{s+\frac 12-\al+} \ll N^{\frac 52-s}M^{-\frac 12-\al+s+} +N^{\frac 32}M^{-2\al+\s+} +N^{\frac 32-\al+} ,\label{Cond1}\\
  \text{(ii)}& \quad N^{1-s}M^{-\frac 12-\al+s} +M^{-2\al+\s+}+N^{-\al+} \ll L^{-1}T^{-1}. \label{Cond2}
\end{align}

Consider first \eqref{Cond1}. We have
\begin{align*}
N^{1-s}M^{s+\frac 12-\al+} \ll N^{\frac 52-s}M^{-\frac 12-\al+s+}  \quad \Longleftarrow \quad M\ll N^{\frac 32+}.
\end{align*}
Then, under $M\ll N^{\frac 32+}$, we find 
\begin{align*}
N^{\frac 32-\al}  \ll N^{\frac 52-s}M^{-\frac 12-\al+s+}
\end{align*}
since $s>0$ and $\al\leq \frac 12$. Therefore, \eqref{Cond1} is satisfied provided that $M\ll N^{\frac 32+}$. We move onto \eqref{Cond2}. As $\s<2\al$, we may choose $M$ large enough so that $M^{2\al-\s-} \gg LT$.
Moreover, if we take $M=N^{k}$ for some $1<k<\frac 32$ to be determined, then 
\begin{align}
N^{1-s}M^{-\frac 12-\al+s} = N^{1-s-k(\frac 12 +\al-s)} \ll L^{-1}T^{-1} \quad \Longleftarrow \quad k>\frac{1-s}{\frac 12+\al-s}. \label{Cond3}
\end{align}
For the condition on $k$ in \eqref{Cond3} to be consistent with $k <\frac 32$ we need
\begin{align*}
\frac{1-s}{\frac 12+\al-s}< \frac 32 \quad \Longleftarrow \quad s<3\al -\frac 12.
\end{align*}
This condition on $s$ is consistent with $\frac 12-\al <s<\al+\frac 12$ provided that $\tfrac 12 -\al <3\al -\tfrac 12$
which forces $\al>\frac 14$.

 Now given any $\e>0$ sufficiently small, we may fix $\frac 14 <\al \leq \frac 12$, $L=\eps^{-1}$, $M=N^{\max(\frac{1-s}{\frac 12+\al-s},1)+}$ and  choose $N=N(\eps,T) \gg 1$ so that \eqref{Cond3} holds. 
 Note that $N\sim (1+\eps T)^{\ta_1}$ for some $\ta_1=\ta_1(\al,s,\s)>0$, namely, it depends polynomially on $T$. Then, by \eqref{smallset}, $\PP( \O^{c}_{\eps^{-1},T,\al,\eps}) <C' e^{-C\e^{-1}} <\e$, we have the bound
 \begin{align*}
\sup_{t\in [0,T]} E(Iw_{M})(t) \leq C(\eps)\jb{T}^{\ta_2}
\end{align*}
for some $\ta_2 >0$.
By \eqref{hstoIh1} and the definition of the modified energy $E(Iw_M)$ we obtain \eqref{bdvk}. 
Setting $\wt{\O}_{T,\eps} =\O_{\eps^{-1},T,\al,\eps}$ completes the proof of Proposition~\ref{prop:bdk}.
\end{proof}

Finally, we have:

\begin{proof}[Proof of Theorem~\ref{THM:ASGWP}] 
We point out that the event $\wt{\O}_{T,\eps}$ controls the forcings $\P_{>M}u_0$ and $v_{M}$ on all of $[0,T]$ and already has small complemental probability. As we then obtain a uniform bound on the growth of $w_{M}$ on $[0,T]$, we conclude that $T_{\max}^{\o}$ from Lemma~\ref{LEM:wM} satisfies $T_{\max}^{\o}\geq T$ for each fixed $\o\in \wt{\O}_{T,\eps}$. 
To establish the almost sure existence statement is a standard argument and one can see for instance \cite[Section 4.3]{FBBM} for details.
\end{proof}

\begin{ackno}\rm
J.F. was partially supported by the ARC project FT230100588.
\end{ackno}

\end{document}